\documentclass[11pt]{amsart}
\usepackage{latexsym,amsmath,amssymb,amsthm,mathrsfs,amscd,enumerate,latexsym,array}
\usepackage[all]{xy}

\theoremstyle{definition}
\newtheorem{Unity}{Unity}[section]
\newtheorem{dfn}[Unity]{Definition}
\newtheorem{rmk}[Unity]{Remark}
\newtheorem{ntt}[Unity]{Notation}
\newtheorem{exam}[Unity]{Example}
\newtheorem{case}{Case}
\newtheorem*{ack}{Acknowledgements}

\theoremstyle{plain}
\newtheorem{thm}[Unity]{Theorem}
\newtheorem{prop}[Unity]{Proposition}
\newtheorem{conj}[Unity]{Conjecture}
\newtheorem{lem}[Unity]{Lemma}
\newtheorem{cor}[Unity]{Corollary}

\begin{document}

\title[Higher order minimal families of rational curves on Fano manifolds]{Higher order minimal families of rational curves and Fano manifolds with nef Chern characters}
\author{Taku Suzuki}
\keywords{Fano manifolds, rational curves, Chern characters, covered by rational manifolds.}
\subjclass[2010]{Primary 14J45, Secondary 14C17, 14M20.}
\address{Department of Mathematics, School of Fundamental Science and Engineering, Waseda University, 3-4-1 Ohkubo, Shinjuku, Tokyo 169-8555, Japan}
\email{s.taku.1231@moegi.waseda.jp}

\maketitle

\begin{abstract}
In this paper, we investigate higher order minimal families $H_i$ of rational curves associated to Fano manifolds $X$. 
We prove that $H_i$ is also a Fano manifold if the Chern characters of $X$ satisfy some positivity conditions. 
We also provide a sufficient condition for Fano manifolds to be covered by higher rational manifolds. 
\end{abstract}

\section{Introduction}

Throughout this paper, we consider a smooth Fano manifold $X$, namely, a smooth complex projective variety $X$ with ample anti-canonical line bundle $-K_X$. 
By S.\ Mori's famous result (\cite{Mo}), $X$ is covered by rational curves, so we consider families $H_1$ of rational curves on $X$ through a fixed general point. 
If $H_1$ has minimal $(-K_X)$-degree among them, we say it is \textit{minimal}, and write $X \vdash H_1$ in this paper. 
It is known that minimal families $H_1$ give much geometrical information concerning the original manifold $X$. 
Many authors have studied the structures of them (for instance, \cite{Ko}, \cite{Ke2}, \cite{Dr}, and so on). 
The image of $H_1$ in the projectivized tangent space by the tangent map (see Definition \ref{Def2}) is called the \textit{variety of minimal rational tangents} (\textit{VMRT}). 
The theory of VMRT's has been built by J.-M.\ Hwang and N.\ Mok (see \cite{Hw} and \cite{HM}), and it has played an important role in many problems regarding higher dimensional algebraic geometry. 

If $H_1$ is also a Fano manifold, we can consider minimal families $H_2$ of rational curves on $H_1$ through a general point. 
In the same way, we  inductively define \textit{higher order minimal families of rational curves} $X \vdash H_1 \vdash \cdots \vdash H_i$. 
In this paper, we focus on higher order minimal families $H_i$, which have been hardly investigated. 
In addition, we introduce two invariants $\underline{N}_X$ and $\overline{N}_X$ as follows: 
$\underline{N}_X$ (resp.\ $\overline{N}_X$) is the minimum (resp.\ maximum) length of chains of families $X \vdash H_1 \vdash \cdots \vdash H_N$ such that $H_N$ is not a Fano manifold (see also Definition \ref{Def11}). 

For Fano manifolds $X$, there are several invariants which show positivity of $-K_X$, for instance, the index $r_X$ and the pseudo-index $i_X$. 
Here $r_X$ is the greatest positive integer $r$ such that $-K_X=rL$ for some line bundle $L$, and $i_X$ is the minimum of intersection numbers of $-K_X$ with rational curves on $X$. 
In general, $r_X \le i_X \le {\dim}X+1$ holds, and equalities hold if and only if $X$ is isomorphic to a projective space (\cite{KO} and \cite{CMSB}). 
These invariants are valuable in many classification problems of higher dimensional manifolds. 
In regard to our invariants, $\underline{N}_X \le \overline{N}_X \le {\dim}X$ holds, and equalities hold if and only if $X$ is isomorphic to a projective space (see Example \ref{Exam3}). 
Thus, $\underline{N}_X$ and $\overline{N}_X$ also seem significant invariants which show positivity of $-K_X$ just like $r_X$ and $i_X$. 

On the other hand, we have a natural question related to classifications of higher Fano manifolds: 
\begin{quote}
Is $\mathbb{P}^n$ the unique Fano $n$-fold whose $k$-th Chern character is weakly positive (see Definition \ref{Def4}) for every $1\le k\le n$?
\end{quote}
That is why it is worth studying Fano manifolds whose Chern characters satisfy some positivity conditions. 
Fano manifolds $X$ with nef ${\rm ch}_2(X)$ were introduced by A.\ J.\ de Jong and J.\ Starr (see \cite{JS}). 
Such manifolds have been investigated by several authors, and classified in case $i_X \ge {\rm dim}X-2$ (\cite{AC2}) and in special toric cases (\cite{No} and \cite{Sa}). 
C.\ Araujo and A.-M.\ Castravet researched Fano manifolds $X$ with weakly positive ${\rm ch}_2(X)$ in \cite{AC1}, and they obtained the following: 
\begin{enumerate}
\item If ${\rm ch}_2(X)$ is weakly positive and there is a family $X \vdash H_1$ of dimension at least $1$, then $H_1$ is a Fano manifold with $\rho_{H_1} \le 2$. (In particular, $\underline{N}_X \ge 2$.)
\item If ${\rm ch}_2(X)$ is weakly positive, ${\rm ch}_3(X)$ is nef, and there is a family $X \vdash H_1$ of dimension at least $2$, then ${\rm ch}_2(H_1)$ is weakly positive, $\rho_{H_1} =1$, and a second family $H_1 \vdash H_2$ has positive dimension. (In particular, $\underline{N}_X \ge 3$.)
\end{enumerate}
Here, we remark that a Fano manifold may possibly have several minimal families, but they have uniform dimension (see Lemma \ref{Lem1}). 

In addition, as higher dimensional versions of S.\ Mori's result, the following results are known: 
\begin{enumerate}
\item If ${\rm ch}_2(X)$ is nef and $i_X \ge 3$, then $X$ is covered by rational surfaces (\cite{JS}). 
\item If ${\rm ch}_2(X)$ is weakly positive, ${\rm ch}_3(X)$ is nef, and there are families $X \vdash H_1 \vdash H_2$ such that ${\rm dim}H_1 \ge 2$ and $H_1, H_2 \not\cong \mathbb{P}^d$, then $X$ is covered by rational 3-folds (\cite{AC1}). 
\end{enumerate}

Our aim of this paper is to generalize these results. 
We will prove the following theorem: 

\begin{thm}\label{Thm1}
Let $N$ be an integer with $2 \le N \le 100$. 
Let $X$ be a Fano manifold with nef Chern characters ${\rm ch}_2(X), \ldots , {\rm ch}_N(X)$. 
Suppose that there is a family $X \vdash H_1$ of dimension at least $N^2-N-1$. 
\begin{enumerate}
\item Then $\underline{N}_X \ge N$ holds. 
\item Furthermore, if $H_1 \not\cong \mathbb{P}^{d}, Q^d$ and there are families $H_1 \vdash H_2 \vdash \cdots \vdash H_{N-1}$ such that $H_2, \ldots , H_{N-1} \not\cong \mathbb{P}^{d}$, then $X$ is covered by rational $N$-folds. 
\end{enumerate}
\end{thm}

This paper is organized as follows. 
In Section 2, we explain several facts on families of rational curves, and introduce higher order minimal families and the invariants $\overline{N}_X$ and $\underline{N}_X$. 
In Section 3, we mention classical results related to the Bernoulli numbers and binomial coefficients. 
In Section 4, we compute the Chern characters of higher order minimal families. 
In Section 5, we prove Theorem \ref{Thm1}. 
The author used a computer for some part of computation in Section 4. 

\begin{ack}
The author would like to express his gratitude to his supervisor Professor Hajime Kaji for valuable discussions and helpful advice. 
The author would also like to thank Professor Yasunari Nagai for beneficial suggestions and comments. 
In particular, he suggested considering the invariants $\overline{N}_X$ and $\underline{N}_X$. 
The author is also grateful to Professor Kiwamu Watanabe for useful comments. 
The author is supported by Grant-in-Aid for Research Activity Start-up from the Japan Society for the Promotion of Science. 
\end{ack}

\section{Families of rational curves}

First, we mention the definitions of positivity of cycles and the Picard number. 

\begin{dfn}\label{Def4}
Let $Y$ be a projective manifold, and $r$ a non-negative integer. 
We denote by $N^r(Y)$ (resp.\ $N_r(Y)$) the group of cycles of codimension $r$ (resp.\ dimension $r$) on $Y$ modulo numerical equivalence, and set $N^r(Y)_{\mathbb{R}}:=N^r(Y) \otimes \mathbb{R}$ and $N_r(Y)_{\mathbb{R}}:=N_r(Y) \otimes \mathbb{R}$. 
If $\alpha \in N^r(Y)_{\mathbb{R}}$ satisfies $\alpha \cdot \beta \ge 0$ (resp.\ $\alpha \cdot \beta > 0$) for every effective integral cycle $\beta \ne 0$ of dimension $r$, we say that $\alpha$ is \textit{nef} (resp.\ \textit{weakly positive}) and write $\alpha \ge 0$ (resp.\ $\alpha >0$). 
For $\alpha_1, \alpha_2 \in N^r(Y)_{\mathbb{R}}$, we also write $\alpha_1 \ge \alpha_2$ (resp.\ $\alpha_1 > \alpha_2$) when $\alpha_1 - \alpha_2 \ge 0$ (resp.\ $\alpha_1 - \alpha_2 > 0$).
We call $\rho_Y := {\rm dim}_{\mathbb{R}}N_1(Y)_{\mathbb{R}}$ the Picard number of $Y$. 
\end{dfn}

We refer to \cite{AC1} and \cite[I and II]{Ko} for basic theory of families of rational curves. 
In this section, we consider a smooth Fano manifold $X$ of dimension $n$. 
Let $x$ be a general point of $X$. 

\begin{dfn}\label{Def1}
We denote by $\text{RatCurves}^n(X,x)$ the normalization of the scheme of rational curves on $X$ passing through $x$ (see \cite[II.2]{Ko}).
An irreducible component $H$ of $\text{RatCurves}^n(X,x)$ is called a \textit{family of rational curves through $x$}. 
Notice that $H$ is smooth because $x$ is a general point. 
When $H$ parametrizes curves of minimal $(-K_X)$-degree among rational curves through $x$, we say that it is \textit{minimal}, and write $X \vdash H$ in this paper. 
\end{dfn}

From now on, let $H$ be a minimal family of rational curves through $x$, and $d$ the dimension of $H$. 
The minimality yields that any curve obtained by a deformation of curves parametrized by $H$ must be irreducible, so $H$ is proper. 

\begin{lem}[{\cite{CMSB} and \cite{Mi}}, see also \cite{Ke1} and \cite{DH}]\label{Lem1}
We have the following: 
\begin{enumerate}
\item $d=(-K_X \cdot C) -2 \le n-1$, where $C$ is a curve parametrized by $H$. 
\item $d=n-1$ holds if and only if $X \cong \mathbb{P}^n$. 
\item When $n \ge 3$ and $\rho_X=1$, $d=n-2$ holds if and only if $X \cong Q^n$. 
\end{enumerate}
\end{lem}

\begin{dfn}\label{Def2}
Let $U$ be the universal family of $H$, and let $\pi : U \rightarrow H$ and $e: U \rightarrow X$ be the associated morphisms. 
We note that $\pi$ is a $\mathbb{P}^1$-bundle and there is a unique section $\sigma : H \rightarrow \Sigma:=\sigma(H) \subset U$ such that $e(\Sigma)=\{x\}$. 
The cokernel of the inclusion
$$\sigma^* T_{\pi} \hookrightarrow \sigma^* e^* T_X \cong T_x X \otimes \mathscr{O}_H$$
is locally free, and it defines a finite morphism $\tau : H \rightarrow \mathbb{P}(T_x X^{\vee})$ (see \cite[Theorem 3.3 and 3.4]{Ke2}). 
We call this morphism $\tau$ the \textit{tangent map}. 
We notice that it is birational onto its image (see \cite{HM}) and sends a curve smooth at $x$ to its tangent direction. 
Let $L$ be the ample line bundle $\tau ^* \mathscr{O}(1)$. 
We call the pair $(H,L)$ a \textit{polarized minimal family of rational curves through $x$}. 
\end{dfn}

\begin{exam}\label{Exam2}
We have the following: 
\begin{enumerate}
\item When $X=\mathbb{P}^n$, $(H,L) \cong (\mathbb{P}^{n-1}, \mathscr{O}_{\mathbb{P}^{n-1}}(1))$. 
\item When $X=Q^n$, $(H,L) \cong (Q^{n-2}, \mathscr{O}_{Q^{n-2}}(1))$. 
\item When $X={\rm Bl}_{S_m}(\mathbb{P}^n)$, where $S_m$ is a linear space of dimension $m$ in $\mathbb{P}^n$, $(H,L) \cong (\mathbb{P}^m, \mathscr{O}_{\mathbb{P}^m}(1))$. 
\end{enumerate}
\end{exam}

\begin{lem}[see {\cite[2.2]{AC1}}]\label{Lem2}
Let $\pi$, $e$, $\Sigma$, and $L$ be as in Definition \ref{Def2}. 
\begin{enumerate}
\item Let $D$ be an $\mathbb{R}$-divisor on $X$, and set $a:=(D \cdot C)$, where $C$ is a curve parametrized by $H$. 
Then $e^* D=a(\Sigma+\pi^*c_1(L))$. 
\item $\Sigma \cdot e^* \alpha =0$ for any cycle $\alpha$ of positive codimension on $X$. 
\end{enumerate}
\end{lem}

\begin{dfn}\label{Def3}
We denote by $H^{\circ}$ the subvariety of $H$ parametrizing curves that is smooth at $x$.  
It is known that $H \backslash H^{\circ}$ is at most finite (see \cite[Theorem 3.3]{Ke2}).  
\end{dfn}

\begin{lem}[{\cite[Lemma 2.3]{AC1}}]\label{Lem3}
Let $L$ and $H^{\circ}$ be as in Definitions \ref{Def2} and \ref{Def3}. 
Suppose that there is a subvariety $Z \subset H^{\circ}$ such that $(Z,L|_Z)\cong (\mathbb{P}^k,\mathscr{O}_{\mathbb{P}^k}(1))$. 
Then there is a generically injective morphism $f: (\mathbb{P}^{k+1},p) \rightarrow (X,x)$ that maps lines through $p$ birationally to curves parametrized by $H$. 
\end{lem}

\begin{dfn}[{\cite[Definition 2.6]{AC1}}]\label{Def5}
Let $\pi$ and $e$ be as in Definition \ref{Def2}. 
For every positive integer $r$, we define a linear map
$$T: N^r(X)_{\mathbb{R}} \rightarrow N^{r-1}(H)_{\mathbb{R}},\ \ \alpha \mapsto \pi_* e^* \alpha.$$
\end{dfn}

\begin{lem}\label{Lem4}
Let $L$ and $T$ be as in Definitions \ref{Def2} and \ref{Def5}. 
Let $D$ be an $\mathbb{R}$-divisor on $X$, and set $a:=(D \cdot C)$, where $C$ is a curve parametrized by $H$, and $\alpha \in N^r(X)_{\mathbb{R}}$. 
\begin{enumerate}
\item $T(D^k)=a^k c_1(L)^{k-1}$. In particular, $T(c_1(X))=d+2$ (see Lemma \ref{Lem1}(1)). 
\item $T(\alpha \cdot D^k)=a^k T(\alpha) \cdot c_1(L)^k$. 
\item If $\alpha \ge 0$ (resp.\ $\alpha >0$), then $T(\alpha) \ge 0$ (resp.\ $T(\alpha) >0$). 
\end{enumerate}
\end{lem}
\begin{proof}
(1) and (3) have been proved in \cite[Lemma 2.7]{AC1}. 
We show (2). 
By using Lemma \ref{Lem2} and the projection formula, we see:
\begin{eqnarray*}
T(\alpha \cdot D^k)&=&\pi_*\biggl\{e^* \alpha \cdot a^k \sum_{i=0}^{k} \binom{k}{i}\Sigma^{k-i} \cdot \pi^*c_1(L)^i \biggr\}\\
&=&a^k \pi_* \Bigl(e^* \alpha \cdot \pi^*c_1(L)^k \Bigr)\\
&=&a^k T(\alpha) \cdot c_1(L)^k.
\end{eqnarray*}
\end{proof}

\begin{lem}[{\cite[Proposition 1.3]{AC1}}]\label{Lem5}
Let $L$ and $T$ be as in Definitions \ref{Def2} and \ref{Def5}, and let $B_m$ be the $m$-th Bernoulli number (see Definition \ref{Def7}). 
For every positive integer $j$, the $j$-th Chern character of $H$ is given by the formula: 
$${\rm ch}_j (H)= \sum_{m=0}^j \frac{(-1)^m B_m}{m!}T({\rm ch}_{j+1-m}(X)) \cdot c_1(L)^m - \frac{1}{j!}c_1(L)^j.$$
\end{lem}

We define \textit{higher order polarized minimal families of rational curves} associated to $X$ as follows: 

\begin{dfn}\label{Def6}
We write $X \vdash H_1 \vdash \cdots \vdash H_N$ when $H_{i-1}$ is a Fano manifold of positive dimension and $H_{i-1} \vdash H_i$ for every $1 \le i \le N$, where $H_0:=X$. 
Let $L_i$ be the polarization associated to $H_i$ as in Definition \ref{Def2}. 
We call $(H_i,L_i)$ an \textit{$i$-th order polarized minimal family of rational curves} associated to $X$. 
For positive integers $m$ and $r$, we define a linear map
$$T^m: N^r(H_i)_{\mathbb{R}} \rightarrow N^{r-m}(H_{i+m})_{\mathbb{R}}$$
as the composition of $T$'s associated to $H_{i+1}, \ldots ,H_{i+m}$ in Definition \ref{Def5}. 
\end{dfn}

\begin{dfn}\label{Def11}
We define the \textit{length} of a chain $X \vdash H_1 \vdash \cdots \vdash H_N$ as $N$, and denote by $\underline{N}_X$ (resp.\ $\overline{N}_X$) the minimum (resp.\ maximum) length of chains of families $X \vdash H_1 \vdash \cdots \vdash H_N$ such that $H_N$ is not a Fano manifold. Clearly, $1 \le \underline{N}_X \le \overline{N}_X \le n$ for any Fano $n$-fold. 
In addition, if $Y$ is not a Fano manifold, we set $\underline{N}_Y = \overline{N}_Y =0$. 
\end{dfn}

\begin{exam}\label{Exam3}
Example \ref{Exam2} yields the following: 
\begin{enumerate}
\item When $X=\mathbb{P}^n$, we have $\underline{N}_X = \overline{N}_X = n$. It is the unique Fano $n$-fold $X$ satisfying $\overline{N}_X = n$ by Lemma \ref{Lem1}. 
\item When $X=Q^n$, we have $\underline{N}_X = \overline{N}_X=\lfloor \frac{n+1}{2} \rfloor$. 
\item When $X={\rm Bl}_{S_m}(\mathbb{P}^n)$, we have $\underline{N}_X = \overline{N}_X=m+1$. 
\item When $X= Q^{m+1} \times \mathbb{P}^m$, we have $\underline{N}_X =\lfloor \frac{m+2}{2} \rfloor$ and $\overline{N}_X=m$. 
\end{enumerate}
\end{exam}

\section{Bernoulli numbers and binomial coefficients}

In this section, we mention well-known classical results related to the Bernoulli numbers and binomial coefficients. 
We refer to \cite[1 and 2]{AIK} for theory of the Bernoulli numbers. 

\begin{dfn}\label{Def7}
The Bernoulli numbers $B_m$'s are defined by the formula: 
$$\frac{t}{e^t-1}=\sum_{m=0}^{\infty} \frac{B_m}{m!} t^m.$$
\end{dfn}

\begin{rmk}\label{Rmk1}
Table \ref{Table5} is the list of the Bernoulli numbers $B_m$'s for $m \le 10$. 
\begin{table}[h]
 \caption{The Bernoulli numbers.}
  \begin{tabular}{c|>{\centering}p{14pt}>{\centering}p{14pt}>{\centering}p{14pt}>{\centering}p{14pt}>{\centering}p{14pt}>{\centering}p{14pt}>{\centering}p{14pt}>{\centering}p{14pt}>{\centering}p{14pt}>{\centering}p{14pt}>{\centering}p{14pt}@{}l@{}} 
 $m$     &  $0$  & $1$  & $2$     & $3$     & $4$   & $5$   & $6$    & $7$  & $8$  & $9$  & $10$&\\[3pt] \hline
$B_m$ & $1$ & $-\frac{1}{2}$ & $\frac{1}{6}$  & $0$ & $-\frac{1}{30}$ & $0$ & $\frac{1}{42}$ & $0$ & $-\frac{1}{30}$ & $0$ & $\frac{5}{66}$&
  \end{tabular}
 \label{Table5}
\end{table}
\end{rmk}

\begin{rmk}
There is another convention on the definition of the Bernoulli numbers $\tilde{B}_m$'s:
$$\frac{te^t}{e^t-1}=\sum_{m=0}^{\infty} \frac{\tilde{B}_m}{m!} t^m.$$
Then notice that $\tilde{B}_m=(-1)^mB_m$. 
\end{rmk}

\begin{lem}\label{Lem6}
We have the following formulas for the Bernoulli numbers:
\begin{eqnarray}
B_m&=&\sum_{p=1}^m \frac{1}{p+1}\sum_{q=1}^p (-1)^q \binom{p}{q}q^m\ \ \ \text{for }m \ge 1.\\
\sum_{r=1}^{q-1} r^j&=& \frac{1}{j+1} \sum_{m=0}^j B_m \binom{j+1}{m} q^{j+1-m}.\\
\sum_{r=1}^{q} r^j&=& \frac{1}{j+1} \sum_{m=0}^j (-1)^m B_m \binom{j+1}{m} q^{j+1-m}.
\end{eqnarray}
\end{lem}

\begin{rmk}
In Lemma \ref{Lem6}, we obtain the formula (1) by viewing the Stirling numbers (see \cite[Theorem 2.8]{AIK}). 
The equations (2) and (3) are often called the Seki-Bernoulli formulas (see \cite[1.1]{AIK}). 
\end{rmk}

\begin{lem}\label{Lem7}
We have the following formulas for binomial coefficients:
\setcounter{equation}{0}
\begin{eqnarray}
\sum_{q=1}^{p} (-1)^q \binom{p}{q} q^m&=&0 \ \ \ \text{for }1 \le m < p.\\
\sum_{q=r+1}^{p} (-1)^q \binom{p}{q} &=& (-1)^{r-1}\binom{p-1}{r}.
\end{eqnarray}
\end{lem}

\begin{rmk}
In Lemma \ref{Lem7}, the formula (1) follows from the fact that $\frac{d^mf}{dt^m}(-1)=0$ for the polynomial $f(t):=(1+t)^p$. 
We can check the formula (2) by using Pascal's rule: $\binom{p-1}{q}+\binom{p-1}{q-1}=\binom{p}{q}$. 
\end{rmk}

\begin{ntt}\label{Def9}
For positive integers $m$ and $p$, we set:
$$c_{(m,p)}:=\sum_{q=1}^p (-1)^q \binom{p}{q}q^m.$$
\end{ntt}

\begin{rmk}\label{Rmk3}
Lemma \ref{Lem7}(1) says that $c_{(m,p)}=0$ for $m<p$, and Lemma \ref{Lem6}(1) yields: 
$$B_m=\sum_{p=1}^m \frac{c_{(m,p)}}{p+1}.$$
Table \ref{Table6} is the list of $c_{(m,p)}$ for small $m$ and $p$. 
\begin{table}[h]
 \caption{$c_{(m,p)}$.}
  \begin{tabular}{c|>{\centering}p{25pt}>{\centering}p{25pt}>{\centering}p{25pt}>{\centering}p{25pt}>{\centering}p{25pt}>{\centering}p{25pt}@{}l@{}} 
 $p$     &  $1$  & $2$     & $3$     & $4$   & $5$   & $6$&      \\[3pt] \hline
$c_{(1,p)}$ & $-1$ & $0$ & $0$ & $0$ & $0$ & $0$&\\[3pt]
$c_{(2,p)}$ & $-1$ & $2$ & $0$ & $0$ & $0$ & $0$&\\[3pt]
$c_{(3,p)}$ & $-1$ & $6$ & $-6$ & $0$ & $0$ & $0$&\\[3pt]
$c_{(4,p)}$ & $-1$ & $14$ & $-36$ & $24$ & $0$ & $0$&\\[3pt]
$c_{(5,p)}$ & $-1$ & $30$ & $-150$ & $240$ & $-120$ & $0$&\\[3pt]
$c_{(6,p)}$ & $-1$ & $62$ & $-540$ & $1560$ & $-1800$ & $720$&
  \end{tabular}
\label{Table6}
\end{table}
\end{rmk}

\section{The Chern characters of higher order minimal families}

In this section, let $X$ be a Fano manifold, and suppose that there is a chain of families $X \vdash H_1 \vdash \cdots \vdash H_N$. 
Let $L_i$ be as in Definition \ref{Def6}. 
The goal of this section is to calculate the Chern characters of $H_i$ by applying Lemma \ref{Lem5}. 

\begin{ntt}\label{Ntt10}
We denote by $d_i$ the dimension of $H_i$, and by $a_i$ the intersection number of $L_{i-1}$ and a curve parametrized by $H_i$. 
\end{ntt}

\begin{exam}\label{Exam1}
We calculate $c_1(H_3)$ under the assumption $a_2=a_3=1$. 
By Lemma \ref{Lem5}, we know:
\begin{eqnarray*}
c_1(H_1)&=&-c_1(L_1)+\frac12 T(c_1(X))c_1(L_1)+T({\rm ch}_2(X)),\\
{\rm ch}_2(H_1)&=&-\frac12 c_1(L_1)^2+\frac{1}{12} T(c_1(X))c_1(L_1)^2+\frac12 T({\rm ch}_2(X))\cdot c_1(L_1)+T({\rm ch}_3(X)),\\
{\rm ch}_3(H_1)&=&-\frac16 c_1(L_1)^3+\frac{1}{12} T({\rm ch}_2(X))\cdot c_1(L_1)^2+\frac12T({\rm ch}_3(X))\cdot c_1(L_1) +T({\rm ch}_4(X)).
\end{eqnarray*}
Since $a_2=1$, Lemma \ref{Lem4} yields the following formulas: 
\begin{eqnarray*}
c_1(H_2)&=&-c_1(L_2)+\frac12 T(c_1(H_1))c_1(L_2)+T({\rm ch}_2(H_1))\\
&=&-2c_1(L_2)+\frac13 T(c_1(X))c_1(L_2)+T^2({\rm ch}_2(X))c_1(L_2)+T^2({\rm ch}_3(X)),\\
{\rm ch}_2(H_2)&=&-\frac12 c_1(L_2)^2+\frac{1}{12} T(c_1(H_1))c_1(L_2)^2+\frac12 T({\rm ch}_2(H_1))\cdot c_1(L_2)+T({\rm ch}_3(H_1))\\
&=&-c_1(L_2)^2+\frac{1}{12} T(c_1(X))c_1(L_2)^2+\frac{5}{12}T^2({\rm ch}_2(X))c_1(L_2)^2\\
&&+T^2({\rm ch}_3(X))\cdot c_1(L_2)+T^2({\rm ch}_4(X)).
\end{eqnarray*}
Furthermore, since $a_3=1$, we have: 
\begin{eqnarray*}
c_1(H_3)&=&-c_1(L_3)+\frac12 T(c_1(H_2))c_1(L_3)+T({\rm ch}_2(H_2))\\
&=&-3c_1(L_3)+\frac14 T(c_1(X))c_1(L_3)+\frac{11}{12}T^2({\rm ch}_2(X))c_1(L_3)\\
&&+\frac32T^3({\rm ch}_3(X))c_1(L_3)+T^3({\rm ch}_4(X)).
\end{eqnarray*}
\end{exam}

We generalize Example \ref{Exam1} as follows: 

\begin{ntt}
For any integers $i\ge 1$, $j \ge 1$, and $0 \le k \le i+j$, we define $\alpha_{(i,j,k)} \in N^j(H_i)_{\mathbb{R}}$ and $b_{(i,j,k)} \in \mathbb{R}$ as follows: 
\begin{eqnarray*}
\alpha_{(i,j,k)}&:=&\begin{cases}c_1(L_i)^j & \text{if}\ \ k=0\\
T^k({\rm ch}_k(X))c_1(L_i)^j & \text{if}\ \ 1 \le k \le i\\
T^i({\rm ch}_k(X))\cdot c_1(L_i)^{i+j-k} & \text{if}\ \ i<k \le i+j\end{cases},\\
b_{(1,j,k)}&:=&\begin{cases}-\frac{1}{j!} & \text{if}\ \ k=0\\
\frac{(-1)^{j+1-k} B_{j+1-k}}{(j+1-k)!} & \text{if}\ \ k \ge 1\end{cases},\\
b_{(i,j,k)}&:=&\sum_{m=0}^{\text{min}\{j,\,i+j-k\}}\frac{(-1)^m B_m}{m!}b_{(i-1,j+1-m,k)}-\begin{cases}\frac{1}{j!} & \text{if}\ \ k=0\\
0 & \text{if}\ \ k \ge 1\end{cases}\ \ \ \text{for }i \ge 2.
\end{eqnarray*}
\end{ntt}

\begin{prop}\label{Prop1}
We assume $a_2=\cdots =a_i=1$. Then we have: 
$${\rm ch}_j(H_i)=\sum_{k=0}^{i+j}b_{(i,j,k)}\alpha_{(i,j,k)}.$$
\end{prop}

\begin{proof}
The case $i=1$ follows from Lemma \ref{Lem5} directly, so we assume $i \ge 2$. 
Then, by Lemma \ref{Lem5} again, 
$${\rm ch}_j(H_i)=\sum_{m=0}^j \frac{(-1)^m B_m}{m!}T({\rm ch}_{j+1-m}(H_{i-1})) \cdot c_1(L_i)^m - \frac{1}{j!}c_1(L_i)^j.$$
The inductive hypothesis yields: 
$$T({\rm ch}_{j+1-m}(H_{i-1})) \cdot c_1(L_i)^m=\sum_{k=0}^{i+j-m}b_{(i-1,j+1-m,k)}T(\alpha_{(i-1,j+1-m,k)})\cdot c_1(L_i)^m.$$
Since $a_i=1$, it is easy to see that
$T(\alpha_{(i-1,j+1-m,k)})\cdot c_1(L_i)^m=\alpha_{(i,j,k)}$ by Lemma \ref{Lem4}. 
Thus, we have: 
\begin{eqnarray*}
{\rm ch}_j(H_i)&=&\sum_{k=0}^{i+j}\biggl\{\sum_{m=0}^{\text{min}\{j,\,i+j-k\}}\frac{(-1)^m B_m}{m!}b_{(i-1,j+1-m,k)} \biggr\}\alpha_{(i,j,k)} - \frac{1}{j!}c_1(L_i)^j\\
&=&\sum_{k=0}^{i+j}b_{(i,j,k)}\alpha_{(i,j,k)}.
\end{eqnarray*}
\end{proof}

\begin{rmk}\label{Rmk2}
Tables \ref{Table1}-\ref{Table4} are the lists of $b_{(i,j,k)}$ for small $i$ and $j$. (The author used a computer.)
\begin{table}[h]
\caption{$b_{(i,1,k)}$.}
\begin{tabular}{c|>{\centering}p{23pt}>{\centering}p{23pt}>{\centering}p{23pt}>{\centering}p{23pt}>{\centering}p{23pt}>{\centering}p{23pt}>{\centering}p{23pt}>{\centering}p{23pt}>{\centering}p{23pt}>{\centering}p{23pt}>{\centering}p{23pt}@{}l@{}} 
 $k$     &  $0$  & $1$  & $2$     & $3$     & $4$   & $5$   & $6$    & $7$     & $8$     & $9$   & $10$&\\[3pt] \hline
$b_{(1,1,k)}$ & $-1$ & $\frac{1}{2}$ & $1$&\\[3pt]
$b_{(2,1,k)}$ & $-2$ & $\frac{1}{3}$ & $1$ & $1$&\\[3pt]
$b_{(3,1,k)}$  & $-3$ & $\frac{1}{4}$ & $\frac{11}{12}$ & $\frac{3}{2}$ & $1$&\\[3pt]
$b_{(4,1,k)}$  & $-4$ & $\frac{1}{5}$ & $\frac{5}{6}$ & $\frac{7}{4}$ & $2$ & $1$&\\[3pt]
$b_{(5,1,k)}$  & $-5$ & $\frac{1}{6}$ & $\frac{137}{180}$ & $\frac{15}{8}$ & $\frac{17}{6}$ & $\frac{5}{2}$ & $1$&\\[3pt]
$b_{(6,1,k)}$  & $-6$ & $\frac{1}{7}$ & $\frac{7}{10}$ & $\frac{29}{15}$ & $\frac{7}{2}$ & $\frac{25}{6}$ & $3$ & $1$&\\[15pt]
 \end{tabular}
\label{Table1}

\caption{$b_{(i,2,k)}$.}
\begin{tabular}{c|>{\centering}p{23pt}>{\centering}p{23pt}>{\centering}p{23pt}>{\centering}p{23pt}>{\centering}p{23pt}>{\centering}p{23pt}>{\centering}p{23pt}>{\centering}p{23pt}>{\centering}p{23pt}>{\centering}p{23pt}>{\centering}p{23pt}@{}l@{}} 
 $k$     &  $0$  & $1$  & $2$     & $3$     & $4$   & $5$   & $6$    & $7$     & $8$     & $9$   & $10$&\\[3pt] \hline
$b_{(1,2,k)}$ & $-\frac{1}{2}$ & $\frac{1}{12}$ & $\frac{1}{2}$ & $1$&\\[3pt]
$b_{(2,2,k)}$ & $-1$ & $\frac{1}{12}$ & $\frac{5}{12}$ & $1$ & $1$&\\[3pt]
$b_{(3,2,k)}$  & $-\frac{3}{2}$ & $\frac{3}{40}$ & $\frac{3}{8}$ & $1$ & $\frac{3}{2}$ & $1$&\\[3pt]
$b_{(4,2,k)}$  & $-2$ & $\frac{1}{15}$ & $\frac{31}{90}$ & $1$ & $\frac{11}{6}$ & $2$ & $1$&\\[3pt]
$b_{(5,2,k)}$  & $-\frac{5}{2}$ & $\frac{5}{84}$ & $\frac{23}{72}$ & $\frac{239}{240}$ & $\frac{25}{12}$ & $\frac{35}{12}$ & $\frac{5}{2}$ & $1$&\\[3pt]
$b_{(6,2,k)}$  & $-3$ & $\frac{3}{56}$ & $\frac{167}{560}$ & $\frac{79}{80}$ & $\frac{547}{240}$ & $\frac{15}{4}$ & $\frac{17}{4}$ & $3$ & $1$&\\[15pt]          
\end{tabular}
\label{Table2}

\caption{$b_{(i,3,k)}$.}
\begin{tabular}{c|>{\centering}p{23pt}>{\centering}p{23pt}>{\centering}p{23pt}>{\centering}p{23pt}>{\centering}p{23pt}>{\centering}p{23pt}>{\centering}p{23pt}>{\centering}p{23pt}>{\centering}p{23pt}>{\centering}p{23pt}>{\centering}p{23pt}@{}l@{}} 
 $k$     &  $0$  & $1$  & $2$     & $3$     & $4$   & $5$   & $6$    & $7$     & $8$     & $9$   & $10$&\\[3pt] \hline
$b_{(1,3,k)}$ & $-\frac{1}{6}$ & $0$ & $\frac{1}{12}$ & $\frac{1}{2}$ & $1$& \\[3pt]
$b_{(2,3,k)}$ & $-\frac{1}{3}$ & $\frac{1}{180}$ & $\frac{1}{12}$ & $\frac{5}{12}$ & $1$ & $1$& \\[3pt]
$b_{(3,3,k)}$  & $-\frac{1}{2}$ & $\frac{1}{120}$ & $\frac{29}{360}$ & $\frac{3}{8}$ & $1$ & $\frac{3}{2}$ & $1$& \\[3pt]
$b_{(4,3,k)}$  & $-\frac{2}{3}$ & $\frac{1}{105}$ & $\frac{7}{90}$ & $\frac{7}{20}$ & $1$ & $\frac{11}{6}$ & $2$ & $1$& \\[3pt]
$b_{(5,3,k)}$  & $-\frac{5}{6}$ & $\frac{5}{504}$ & $\frac{227}{3024}$ & $\frac{1}{3}$ & $\frac{721}{720}$ & $\frac{25}{12}$ & $\frac{35}{12}$ & $\frac{5}{2}$ & $1$& \\[3pt]
$b_{(6,3,k)}$  & $-1$ & $\frac{5}{504}$ & $\frac{73}{1008}$ & $\frac{607}{1890}$ & $\frac{241}{240}$ & $\frac{329}{144}$ & $\frac{15}{4}$ & $\frac{17}{4}$ & $3$ & $1$&
\end{tabular}
\label{Table3}
\end{table}

\begin{table}[h]
\caption{$b_{(i,4,k)}$.}
\begin{tabular}{c|>{\centering}p{23pt}>{\centering}p{23pt}>{\centering}p{23pt}>{\centering}p{23pt}>{\centering}p{23pt}>{\centering}p{23pt}>{\centering}p{23pt}>{\centering}p{23pt}>{\centering}p{23pt}>{\centering}p{23pt}>{\centering}p{23pt}@{}l@{}} 
 $k$     &  $0$  & $1$  & $2$     & $3$     & $4$   & $5$   & $6$    & $7$     & $8$     & $9$   & $10$&\\[3pt] \hline
$b_{(1,4,k)}$ & $-\frac{1}{24}$ & $-\frac{1}{720}$ & $0$ & $\frac{1}{12}$ & $\frac{1}{2}$ & $1$& \\[3pt]
$b_{(2,4,k)}$ & $-\frac{1}{12}$ & $-\frac{1}{720}$ & $\frac{1}{240}$ & $\frac{1}{12}$ & $\frac{5}{12}$ & $1$ & $1$& \\[3pt]
$b_{(3,4,k)}$ & $-\frac{1}{8}$  & $-\frac{1}{1120}$ & $\frac{1}{160}$ & $\frac{19}{240}$ & $\frac{3}{8}$ & $1$ & $\frac{3}{2}$ & $1$& \\[3pt]
$b_{(4,4,k)}$ & $-\frac{1}{6}$  & $-\frac{1}{2520}$ & $\frac{113}{15120}$ & $\frac{3}{40}$ & $\frac{251}{720}$ & $1$ & $\frac{11}{6}$ & $2$ & $1$& \\[3pt]
$b_{(5,4,k)}$ & $-\frac{5}{24}$  & $0$ & $\frac{25}{3024}$ & $\frac{865}{12096}$ & $\frac{95}{288}$ & $1$ & $\frac{25}{12}$ & $\frac{35}{12}$ & $\frac{5}{2}$ & $1$& \\[3pt]
$b_{(6,4,k)}$ & $-\frac{1}{4}$  & $\frac{1}{3360}$ & $\frac{887}{100800}$ & $\frac{277}{4032}$ & $\frac{19097}{60480}$ & $1$ & $\frac{137}{60}$ & $\frac{15}{4}$ & $\frac{17}{4}$ & $3$ & $1$ &
\end{tabular}
\label{Table4}
\end{table}

First, we focus on the case $k=0$, then we immediately find the following formulas for $i \le 6$: 
\begin{eqnarray*}
b_{(i,1,0)}&=&-i,\\
b_{(i,2,0)}&=&-\frac{i}{2},\\
b_{(i,3,0)}&=&-\frac{i}{6},\\
b_{(i,4,0)}&=&-\frac{i}{24}.
\end{eqnarray*}
So we can expect: 
$$b_{(i,j,0)}=-\frac{i}{j!}.$$
Next, we view the case $k=1$, then it is easy to check the following equations for $i \le 6$: 
\begin{eqnarray*}
b_{(i,1,1)}&=&\frac{1}{i+1},\\
b_{(i,2,1)}&=&\frac{i}{2(i+1)(i+2)}=\frac{1}{2!}\biggl(\frac{-1}{i+1}+\frac{2}{i+2}\biggr),\\
b_{(i,3,1)}&=&\frac{i^2-i}{6(i+1)(i+2)(i+3)}=\frac{1}{3!}\biggl(\frac{1}{i+1}+\frac{-6}{i+2}+\frac{6}{i+3}\biggr),\\
b_{(i,4,1)}&=&\frac{i^3-5i^2}{24(i+1)(i+2)(i+3)(i+4)}=\frac{1}{4!}\biggl(\frac{-1}{i+1}+\frac{14}{i+2}+\frac{-36}{i+3}+\frac{24}{i+4}\biggr).
\end{eqnarray*}
By comparing with Table \ref{Table6} in Remark \ref{Rmk3}, we can expect: 
$$b_{(i,j,1)}=\frac{(-1)^j}{j!}\sum_{p=1}^j\frac{c_{(j,p)}}{i+p}.$$
We will prove the above formulas in Proposition \ref{Prop2}. 
\end{rmk}

\begin{prop}\label{Prop2}
For any positive integers $i$ and $j$, 
\setcounter{equation}{0}
\begin{eqnarray}
b_{(i,j,0)}&=&-\frac{i}{j!},\\
b_{(i,j,1)}&=&\frac{(-1)^j}{j!}\sum_{p=1}^j\frac{c_{(j,p)}}{i+p}.
\end{eqnarray}
\end{prop}

\begin{proof}
First, we show (1) by induction on $i$. 
The case $i=1$ follows from the definition of $b_{(1,j,0)}$ directly, so we assume $i \ge 2$. 
Then, by the inductive hypothesis and Lemma \ref{Lem6}(3), we have:
\begin{eqnarray*}
b_{(i,j,0)}&=&\sum_{m=0}^j\frac{(-1)^m B_m}{m!}b_{(i-1,j+1-m,0)}-\frac{1}{j!}\\
&=&-\sum_{m=0}^j\frac{(-1)^m B_m}{m!}\cdot\frac{i-1}{(j+1-m)!}-\frac{1}{j!}\\
&=&-\frac{i-1}{j!}\cdot \frac{1}{j+1}\sum_{m=0}^j(-1)^m B_m\binom{j+1}{m}-\frac{1}{j!}\\
&=&-\frac{i}{j!}.
\end{eqnarray*}

Next, we prove (2) by induction on $i$ again. 
In case $i=1$, by Remark \ref{Rmk3}, 
$$b_{(1,j,1)}=\frac{(-1)^j B_j}{j!}=\frac{(-1)^j}{j!}\sum_{p=1}^j\frac{c_{(j,p)}}{1+p}.$$
We suppose $i \ge 2$. 
Then the inductive hypothesis implies: 
\begin{eqnarray*}
b_{(i,j,1)}&=&\sum_{m=0}^j\frac{(-1)^m B_m}{m!}b_{(i-1,j+1-m,1)}\\
&=&\sum_{m=0}^j\frac{(-1)^m B_m}{m!}\cdot\frac{(-1)^{j+1-m}}{(j+1-m)!}\sum_{p=1}^{j+1-m}\frac{c_{(j+1-m,p)}}{i-1+p}.
\end{eqnarray*}
By Lemma \ref{Lem7}(1), 
$$\sum_{p=1}^{j+1-m}\frac{c_{(j+1-m,p)}}{i-1+p}=\sum_{p=1}^{j+1}\frac{c_{(j+1-m,p)}}{i-1+p}=\sum_{p=1}^{j+1}\frac{1}{i-1+p}\sum_{q=1}^p (-1)^q \binom{p}{q}q^{j+1-m}.$$
Thus, by Lemmas \ref{Lem6}(2) and \ref{Lem7}(2), we have:
\begin{eqnarray*}
b_{(i,j,1)}&=&\frac{(-1)^{j+1}}{j!}\sum_{p=1}^{j+1}\frac{1}{i-1+p}\sum_{q=1}^p (-1)^q \binom{p}{q}\frac{1}{j+1}\sum_{m=0}^j B_m\binom{j+1}{m} q^{j+1-m}\\
&=&\frac{(-1)^{j+1}}{j!}\sum_{p=2}^{j+1}\frac{1}{i-1+p}\sum_{q=2}^p (-1)^q \binom{p}{q}\sum_{r=1}^{q-1}r^j\\
&=&\frac{(-1)^{j+1}}{j!}\sum_{p=2}^{j+1}\frac{1}{i-1+p}\sum_{r=1}^{p-1}r^j\sum_{q=r+1}^p (-1)^q \binom{p}{q}\\
&=&\frac{(-1)^{j}}{j!}\sum_{p=2}^{j+1}\frac{1}{i-1+p}\sum_{r=1}^{p-1}r^j(-1)^r \binom{p-1}{r}\\
&=&\frac{(-1)^{j}}{j!}\sum_{p=1}^{j}\frac{c_{(j,p)}}{i+p}.
\end{eqnarray*}
\end{proof}

Recall that $T(c_1(X))=d_1+2$ by Lemma \ref{Lem4}(1). 
Therefore, Propositions \ref{Prop1} and \ref{Prop2} induce the following formulas:

\begin{cor}\label{Cor1}
We assume $a_2=\cdots=a_i=1$. Then we have: 
\setcounter{equation}{0}
\begin{eqnarray}
c_1(H_i)&=&\biggl\{-i+\frac{d_1+2}{i+1}\biggr\}c_1(L_i) +\sum_{k=2}^{i+1}b_{(i,1,k)}\alpha_{(i,1,k)},\\
{\rm ch}_2(H_i)&=&\biggl\{-\frac{i}{2}+\frac{i(d_1+2)}{2(i+1)(i+2)}\biggr\}c_1(L_i)^2+\sum_{k=2}^{i+2}b_{(i,2,k)}\alpha_{(i,2,k)}.
\end{eqnarray}
\end{cor}

We will use the following proposition in order to show the assumption of Corollary \ref{Cor1}. 

\begin{prop}\label{Prop3}
Let $X$ be a Fano manifold, and $(H_1, L_1)$ a polarized minimal family associated to $X$. 
\begin{enumerate}
\item Suppose that  ${\rm ch}_2(X) \ge 0$ and $d_1 \ge 1$. Then $H_1$ is a Fano manifold. (See also \cite[Theorem 1.4(2)]{AC1})
\item Suppose that ${\rm ch}_2(X) > 0$, $d_1 \ge 1$, and $H_1 \not\cong \mathbb{P}^{d_1}$. 
Then $a_2=1$ for every second family $H_1 \vdash H_2$. (See also \cite[Lemma 4.5(1)]{AC1}.)
\item Suppose that ${\rm ch}_2(X) \ge 0$, $d_1 \ge 5$, and $H_1 \not\cong \mathbb{P}^{d_1}, Q^{d_1}$. 
Then $a_2=1$ for every second family $H_1 \vdash H_2$.  
\end{enumerate}
\end{prop}

\begin{proof}
We assume ${\rm ch}_2(X) \ge 0$ and $d_1 \ge 1$. 
By Lemma \ref{Lem5} (or Corollary \ref{Cor1}), 
$$c_1(H_1)=\frac{d_1}{2}c_1(L_1)+T({\rm ch}_2(X)).$$
Hence, Lemma \ref{Lem4}(3) yields $c_1(H_1)>0$, so (1) holds. 
By applying Lemma \ref{Lem4}(1), 
$$d_2+2=T(c_1(H_1))=\frac{d_1}{2}a_2+T^2({\rm ch}_2(X)),$$
so if $a_2 \ge 2$, then $d_2 \ge d_1-2$. 
Furthermore, if ${\rm ch}_2(X) > 0$, then Lemma \ref{Lem4}(3) yields $d_2 >d_1-2$, so $H_1 \cong \mathbb{P}^{d_1}$ by Lemma \ref{Lem1}. 
Thus, (2) holds. 

Next, to show (3), we assume $d_1 \ge 5$. 
According to Lemma \ref{Lem1}, we only have to show that $\rho_{H_1}=1$ under the assumption $a_2 \ge 2$. 
It follows from $d_2 \ge d_1-2$ that ${\rm Locus}(H_2):= e( \pi^{-1}(H_2))$ (see Definition \ref{Def2}) has dimension at least $d_1-1$ in $H_1$. 
Since $H_2$ is a minimal family of rational curves on $H_1$ through a general point, every curve contained in ${\rm Locus}(H_2)$ is numerically proportional to a curve parametrized by $H_2$ (see \cite[Lemma 4.1]{ACO}). 
On the other hand, for any rational curve $C$ on $H_1$, 
$$(-K_{H_1}\cdot C) \ge \frac{d_1}{2}(L_1\cdot C) \ge  \frac{5}{2},$$
so $i_{H_1}\ge 3$. 
Hence, by \cite[Theorem 1.2]{Ca}, for every prime divisor $D \subset H_1$, $N_1(H_1)$ is spanned by numerical classes of curves contained in $D$. 
Thus, we obtain $\rho_{H_1}=1$. 
\end{proof}

From the computation in Remark \ref{Rmk2}, we naturally conjecture the following statement: 

\begin{conj}\label{Conj1}
For $i\ge 1$, $j =1,2$, and $1 \le k\le i+j$, $b_{(i,j,k)} > 0$ holds. 
\end{conj}

\begin{rmk}\label{Rmk4}
Conjecture \ref{Conj1} holds for $k=1$ by Proposition \ref{Prop2}. 
However, it seems difficult to prove it for general $k$. 
The author has checked Conjecture \ref{Conj1} for $i <100$ by using a computer. 
\end{rmk}

\section{Proof of Theorem}

In this section, we prove Theorem \ref{Thm1}. 
By Remark \ref{Rmk4}, it suffices to prove the following theorem:

\begin{thm}\label{Thm2}
Let $N \ge 2$ be an integer such that Conjecture \ref{Conj1} holds for $i < N$. 
Let $X$ be a Fano manifold with nef Chern characters ${\rm ch}_2(X), \ldots , {\rm ch}_N(X)$. 
Suppose that there is a family $X \vdash H_1$ of dimension at least $N^2-N-1$. 
\begin{enumerate}
\item Then $\underline{N}_{H_1} \ge N-1$ holds. 
\item Furthermore, if $H_1 \not\cong \mathbb{P}^{d}, Q^d$ and there are families $H_1 \vdash H_2 \vdash \cdots \vdash H_{N-1}$ such that $H_2, \ldots , H_{N-1} \not\cong \mathbb{P}^{d}$, then $X$ is covered by rational $N$-folds. 
\end{enumerate}
\end{thm}

\begin{proof}[Proof of (1)]
We have already showed the case $N=2$ in Proposition \ref{Prop3}(1), so we may assume $N \ge 3$. 
First, if $H_1\cong Q^d$, then Example \ref{Exam3}(2) implies: 
$$\underline{N}_{H_1} \ge \frac{d}{2} \ge \frac{N^2-N-1}{2} > N-1.$$
Thus, we may assume $H_1 \not\cong  Q^d$. 

Suppose that there is a chain of families $H_1 \vdash H_2 \vdash \cdots \vdash H_M$ for some integer $2 \le M < N$, then we prove that $H_M$ is a Fano manifold of positive dimension. 
Let $d_1, \ldots , d_M$ and $a_2, \ldots, a_M$ be as in Notation \ref{Ntt10}. 
Since $M < N$, we have $d_1 \ge (M+1)^2-(M+1)-1=M^2+M-1$. 
In connection with Corollary \ref{Cor1}, we know: 
\begin{eqnarray*}
\sum_{k=2}^{i+1}b_{(i,1,k)}\alpha_{(i,1,k)} \ge 0 &\text{ for}&1 \le i \le M,\\
\sum_{k=2}^{i+2}b_{(i,2,k)}\alpha_{(i,2,k)} \ge 0 &\text{ for}&1 \le i \le M-1,
\end{eqnarray*}
from the assumptions and Lemma \ref{Lem4}(3).

\begin{case}
$H_1, \ldots , H_M \not\cong \mathbb{P}^{d}$. 
\end{case}

First, we show $a_i=1$ and ${\rm ch}_2(H_{i-1}) >0$ for every $2 \le i \le M$ by induction on $i$. 
Since $d_1 \ge 3^2-3-1 \ge 5$, Proposition \ref{Prop3}(3) implies $a_2=1$, and we have: 
$${\rm ch}_2(H_1) \ge \biggl\{-\frac{1}{2}+\frac{d_1+2}{12}\biggr\}c_1(L_1)^2 \ge \frac{1}{12} c_1(L_1)^2>0.$$
So, we suppose $3 \le i \le M$. 
Then $a_i=1$ follows from the inductive hypothesis ${\rm ch}_2(H_{i-2}) >0$ and Proposition \ref{Prop3}(2). 
In addition, the inductive hypothesis $a_2= \cdots =a_{i-1}=1$ and Corollary \ref{Cor1}(2) yield: 
\begin{eqnarray*}
{\rm ch}_2(H_{i-1}) &\ge& \biggl\{-\frac{i-1}{2}+\frac{(i-1)(d_1+2)}{2i(i+1)}\biggr\}c_1(L_{i-1})^2 \\
&\ge& \frac{i-1}{2} \biggl\{ -1+ \frac{M^2+M+1}{i^2+i} \biggr\} c_1(L_{i-1})^2>0.
\end{eqnarray*}

Next, we show $d_M \ge1$ and $c_1(H_M)>0$. 
Corollary \ref{Cor1}(1) and Lemma \ref{Lem4}(1) imply: 
$$d_M=T(c_1(H_{M-1})) -2 \ge -(M-1)+\frac{d_1+2}{M} -2 \ge \frac{1}{M} >0.$$
Finally, by Corollary \ref{Cor1}(1) again, we obtain:
$$c_1(H_M) \ge \biggl\{-M+\frac{d_1+2}{M+1}\biggr\}c_1(L_M) \ge \frac{1}{M+1}c_1(L_M) >0.$$

Since ${\rm ch}_2(H_{M-1}) >0$, we also obtain $a_{M+1}=1$ for every $(M+1)$-th family $H_M \vdash H_{M+1}$ by Proposition \ref{Prop3}(2). 
We will use this fact in the proof of Case 2 and Theorem \ref{Thm2}(2). 

\begin{case}
$H_1, \ldots , H_{i-1} \not\cong \mathbb{P}^{d}$, and $H_i \cong \mathbb{P}^{d_i}$ for some $1 \le i \le M$.
\end{case}

In this case, by the proof of Case 1, we know $a_2=\cdots =a_i=1$, so Corollary \ref{Cor1}(1) and Lemma \ref{Lem4} yield: 
$$d_i=T(c_1(H_{i-1}))-2 \ge -(i-1)+\frac{d_1+2}{i}-2\ge \frac{M^2}{i}-i+\frac{M+1}{i} -1> M-i.$$
Therefore, by Example \ref{Exam2}(1), $H_M$ is a projective space of positive dimension, in particular, it is a Fano manifold. 
\end{proof}

\begin{proof}[Proof of (2)]
We can take an $N$-th family $H_{N-1} \vdash H_N$ by (1). 
We denote by $\tau_i :H_i \rightarrow P_{i-1}:=\mathbb{P}(T_{x_{i-1}} {H_{i-1}}^{\vee})$ the tangent map, where $x_{i-1}$ is a general point of $H_{i-1}$. 
Let $a_2, \ldots, a_N$ be as in Notation \ref{Ntt10}. 
By the proof of Case 1 in (1), we know $a_2=\cdots =a_N=1$. 
Hence, for $2 \le i \le N$, every curve $C$ parametrized by $H_i$ is a smooth rational curve on $H_{i-1}$. 
Indeed, $\tau_{i-1} (C)$ is a line on $P_{i-2}$. 
In particular, ${H_i}^{\circ} =H_i$ for $2 \le i \le N$ (see Definition \ref{Def3}).

First, we take a rational curve $C$ parametrized by $H_N$. 
As shown, $a_N=1$ implies $(C, L_{N-1}|_C)\cong(\mathbb{P}^1, \mathscr{O}_{\mathbb{P}^1}(1))$ in $H_{N-1}={H_{N-1}}^{\circ} $ (when $N \ge 3$). 
Hence, Lemma \ref{Lem3} induces a generically injective morphism $f_{N-1}: (\mathbb{P}^2,p) \rightarrow (H_{N-2},x_{N-2})$ that maps lines through $p$ birationally to curves parametrized by $H_{N-1}$. 

Next, $a_{N-1}=1$ yields $(\tau_{N-2}\circ f_{N-1})^*\mathscr{O}_{P_{N-3}}(1) \cong \mathscr{O}_{\mathbb{P}^2}(1)$, so the image of $\tau_{N-2}\circ f_{N-1}$ is a $2$-plane in $P_{N-3}$. 
Thus, we have $(f_{N-1}(\mathbb{P}^2), L_{N-2}|_{f_{N-1}(\mathbb{P}^2)}) \cong (\mathbb{P}^2, \mathscr{O}_{\mathbb{P}^2}(1))$ in $H_{N-2}={H_{N-2}}^{\circ}$ (when $N \ge 4$). 
By Lemma \ref{Lem3} again, we get a generically injective morphism $f_{N-2}: (\mathbb{P}^3,p) \rightarrow (H_{N-3},x_{N-3})$ that maps lines through $p$ birationally to curves parametrized by $H_{N-2}$. 

Recall that ${H_i}^{\circ}=H_i$ for $i \ge 2$. 
After the same steps, we finally obtain a generically injective morphism $f_2: (\mathbb{P}^{N-1},p) \rightarrow (H_1,x_1)$ that maps lines through $p$ birationally to curves parametrized by $H_2$, and $(f_2(\mathbb{P}^{N-1}), L_1|_{f_2(\mathbb{P}^{N-1})}) \cong (\mathbb{P}^{N-1}, \mathscr{O}_{\mathbb{P}^{N-1}}(1))$ in $H_1$. 
In other wards, for a general point $h \in H_1$, we get a subvariety $R_h \subset H_1$ such that $h \in R_h$ and $(R_h, L_1|_{R_h}) \cong (\mathbb{P}^{N-1}, \mathscr{O}_{\mathbb{P}^{N-1}}(1))$

To apply Lemma \ref{Lem3}, we show that $R_h \subset {H_1}^{\circ}$ for some $h \in H_1$. 
We assume by contradiction that $R_h \not\subset {H_1}^{\circ}$ for a general point $h \in H_1$. 
Then, since $H_1 \backslash {H_1}^{\circ}$ is finite, there is a point $h_0 \in H_1 \backslash {H_1}^{\circ}$ such that $h_0 \in R_h$ for a general point $h \in H_1$. 
Since $a_2=1$, this implies that $H_1$ is covered by rational curves of $L_1$-degree $1$ through $h_0$, so $(H_1,L_1)$ must be isomorphic to $(\mathbb{P}^{d_1}, \mathscr{O}_{\mathbb{P}^{d_1}}(1))$. 

Therefore, Lemma \ref{Lem3} gives a generically injective morphism $f_1: (\mathbb{P}^N,p) \rightarrow (X,x)$ for a general point $x \in X$. 
Thus, we conclude that $X$ is covered by rational $N$-folds. 
\end{proof}


\begin{thebibliography}{1}
\bibitem{ACO}
M.\ Andreatta, E.\ Chierici, and G.\ Occhetta, {Generalized Mukai conjecture for special Fano varieties}, \textit{Cent. Eur. J. Math.}, \textbf{2}(2) (2004), 272-293.

\bibitem{AC1}
C.\ Araujo and A-M.\ Castravet, {Polarized minimal families of rational curves and higher Fano manifolds}, \textit{American J.\ Math.}, \textbf{134}(1) (2012), 87-107. 

\bibitem{AC2}
C.\ Araujo and A-M.\ Castravet, {Classification of 2-Fano manifolds with high index}, \textit{A Celebration of Algebraic Geometry}, \textbf{18} (2013), 1-36.

\bibitem{AIK}
T.\ Arakawa, T.\ Ibukiyama, and M.\ Kaneko, \textit{Bernoulli Numbers and Zeta Functions}, Springer Monographs in Mathematics, Springer (2014).

\bibitem{Ca}
C.\ Casagrande, {On the Picard number of divisors in Fano manifolds}, \textit{Annales Scientifiques de l'\'Ecole Normale Sup\'erieure}, \textbf{45}(3) (2012), 363-403.

\bibitem{CMSB}
K.\ Cho, Y.\ Miyaoka, and N.\ I.\ Shepherd-Barron, {Characterizations of projective space and applications to complex symplectic manifolds}, \textit{Adv. Stud. Pure Math.}, \textbf{35} (Mathematical Society of Japan, Tokyo, 2002), 1-88. 

\bibitem{DH}
T.\ Dedieu and A.\ H{\"o}ring, {Numerical characterisation of quadrics}, to appear in \textit{Algebraic Geometry}. 

\bibitem{JS}
A.\ J.\ de Jong and J.\ Starr, {Higher Fano manifolds and rational surfaces}, \textit{Duke Math.\ J.}, \textbf{139}(1) (2007), 173-183.

\bibitem{Dr}
S.\ Druel, {Classes de Chern des vari\'et\'es unir\'egl\'ees}, \textit{Math.\ Ann.}, \textbf{335}(4) (2006), 917-935. 

\bibitem{Hw}
J.-M.\ Hwang, {Geometry of minimal rational curves on Fano manifolds}, in \textit{School on Vanishing Theorems and Effective Results in Algebraic Geometry} (Trieste,\ 2000), ICTP Lect.\ Notes, vol.\ 6, 335-393. Abdus Salam Int. Cent. Theoret. Phys. (2001).

\bibitem{HM}
J.\ M.\ Hwang and N.\ Mok, {Birationality of the tangent map for minimal rational curves}, \textit{Asian J.\ Math.}, \textbf{8}(1) (2004), 51-64. 

\bibitem{Ke1}
S.\ Kebekus, {Characterizing the projective space after Cho, Miyaoka and Shepherd-Barron}, \textit{Complex Geometry (G{\"o}ttingen, 2000)}, Springer-Verlag, Berlin (2002), 147-155.

\bibitem{Ke2}
S.\ Kebekus, {Families of singular rational curves}, \textit{J.\ Algebraic Geom.}, \textbf{11}(2) (2002), 245-256. 

\bibitem{KO}
S.\ Kobayashi and T.\ Ochiai, {Characterizations of complex projective spaces and hyperquadrics},
\textit{J.\ Math.\ Kyoto Univ.}, \textbf{13} (1973) 31-47.

\bibitem{Ko}
J.\ Koll\'ar, \textit{Rational Curves on Algebraic Varieties}, Ergeb. Math. Grenzgeb., vol. 32. Springer, Berlin (1996).

\bibitem{Mi}
Y.\ Miyaoka. {Numerical characterizations of hyperquadrics}, \textit{Adv. Stud. Pure Math.}, \textbf{42}, (Mathematical Society of Japan, Tokyo, 2004), 209-235.

\bibitem{Mo}
S.\ Mori, {Projective manifolds with ample tangent bundles}, \textit{Ann. Math.(2)}, \textbf{110}(3) (1979), 593-606. 

\bibitem{No}
E.\ Nobili, {Classification of Toric 2-Fano 4-folds}, \textit{Bull.\ Braz.\ Math.\ Soc.}, New
Series \textbf{42} (2011), 399-414.

\bibitem{Sa}
H.\ Sato, {The numerical class of a surface on a toric manifold}, \textit{Int.\ J.\ Math.\ Math.\ Sci.}, \textbf{2012} (2012), ID:536475.
\end{thebibliography}
\end{document}